\documentclass[12pt]{amsart}
\usepackage{graphicx,amssymb}
\usepackage{mathrsfs}
\usepackage{amssymb,amsmath,amsthm,color}
\usepackage{graphicx,mcite}
\usepackage{hyperref}
\usepackage{url}
\usepackage{setspace}

\theoremstyle{definition}

\theoremstyle{remark}

\numberwithin{equation}{section}
\newcommand{\set}[1]{\left\{#1\right\}}
\newcommand{\suchthat}{\ensuremath{\, \vert \,}}
\newcommand{\R}{\mathbb R}
\newcommand{\C}{\mathbb C}
\newcommand{\Z}{\mathbb Z}
\newcommand{\LL}{ {\mathcal L} }

\def\TagOnRight

\def\R {\mathbb{R}}
\newcommand{\be}{\begin{equation}}
\newcommand{\ee}{\end{equation}}
\newcommand{\bea}{\begin{eqnarray}}
\newcommand{\eea}{\end{eqnarray}}
\newcommand{\Bea}{\begin{eqnarray*}}
\newcommand{\Eea}{\end{eqnarray*}}
\newcommand{\bt}{\begin{Theorem}}
\newcommand{\et}{\end{Theorem}}

\newcommand{\bpr}{\begin{Proposition}}
\newcommand{\epr}{\end{Proposition}}

\newcommand{\bl}{\begin{Lemma}}
\newcommand{\el}{\end{Lemma}}
\newcommand{\bi}{\begin{itemize}}
\newcommand{\ei}{\end{itemize}}

\newtheorem{Definition}{Definition}[section]
\newtheorem{Theorem}[Definition]{Theorem}
\newtheorem{Lemma}[Definition]{Lemma}
\newtheorem{Proposition}[Definition]{Proposition}

\begin{document}
\baselineskip16pt

\title[]{ Nonlinear Schr\"{o}dinger equation for the twisted Laplacian in the critical case}
\author{Vijay Kumar Sohani}%
\address {Harish-Chandra Research Institute, Allahabad-211019
 India}
 \email { sohani@hri.res.in}
%\thanks{%
\subjclass[2010]{Primary 42B37, Secondary 35G20, 35G25}
%\date{\today}
\keywords{ Twisted Laplacian(special Hermite operator), Nonlinear Schr\"{o}dinger equation, 
Strichartz estimates, well posedness }%

%\date{\today}%
%\dedicatory{}%
%\commby{}%
% ----------------------------------------------------------------
\begin{abstract} In \cite{pkrvks} and \cite{pkrvks1}, we prove well-posedness of solution to the nonlinear Schr\"{o}dinger equation associated 
to the twisted Laplacian on $\C^n$ for a  general class of 
nonlinearities including power type
with subcritical case $0\leq \alpha<\frac{2}{n-1}$. In this paper,
we consider critical case $\alpha=\frac{2}{n-1}$ with $n\geq 2$. Our approach is based on truncation of the given nonlinearity $G$, which is used in \cite{CW1}. We obtain solution for the truncated problem. We obtain solution to the 
original problem by passing to the limit.
  \end{abstract}

\maketitle
% ----------------------------------------------------------------
\section{Introduction} 
We consider the initial value problem for the nonlinear 
Schr\"{o}dinger equation for the twisted Laplacian 
$\LL$ : \bea \label{cmainpde}i\partial_tu(z,t)\!\!\!\!& -\!\!\!&\LL
u(z,t)=G(z, u), ~~~~~~~~~~~~z \in \C^n, ~t \in \R \\
\label{cinidata}&& u(z,t_0)= f(z). \eea
Here we consider the nonlinearity $G$ of the form
\bea\label{nlc1}G(z,w) = \psi(x,y,|w|)\, w, ~ (x,y,w)
\in \R^n\times \R^n  \times \C,\eea where $z=x+iy
\in\C^n,w\in\C $ and $\psi
\in C(\R^n\times \R^n \times [0,\infty))\cap C^1(\R^n\times 
\R^n \times (0,\infty))$ satisfy the following inequality
\bea \label{nlc2} |F(x,y,\eta)| \leq C|\eta|^ {\alpha}
\eea 
 with $F= \psi, \partial_{x_j} \psi $, $ \partial_{y_j}\psi$ 
($1\leq j \leq n$) and $\eta\partial_{\eta}\psi(x,y, \eta)$, $\alpha=\frac{2}{n-1}$ with $n\geq 2$ and for some constant $C$. 

In \cite{pkrvks1}, we consider subcritical case $0\leq \alpha< \frac{2}{n-1}$ 
for initial value in $\tilde{W}_{\LL}^{1,2}(\C^n)$. Sobolev space 
$\tilde{W}_{\LL}^{1,p}(\C^n)$ is introduced in \cite{pkrvks1}.
In this paper, we consider the critical case $\alpha=\frac{2}{n-1}$. In subcritical case
$0\leq \alpha< \frac{2}{n-1}$ for each $\alpha$, we have some $q>2$ such that $(q,2+\alpha)$ be an admissible pair (see Definition 3.1 in \cite{pkrvks1}), which is not the case when $\alpha=\frac{2}{n-1}$. We overcome this difficulty by considering admissible 
pair $(\gamma,\rho)$ and by using embedding theorem (Lemma \ref{embedding1}), where 
\Bea \rho=\frac{2n^2}{n^2-n+1}, ~~ \gamma=\frac{2n}{n-1}.
\Eea
To treat the critical case, we adopt truncation argument method of Cazenave 
and Weissler \cite{CW1}. 
To prove local existence, we truncate the given nonlinearity $G$ 
and obtain solution for the truncated problem. Now we obtain solution $u$ for given nonlinearity 
$G$ by using Strichartz estimates and by passing to the limit.

The twisted Laplacian operator 
$\LL$  was introduced by R. S. Strichartz \cite{S1}, and called the  
special Hermite operator.
The twisted Laplacian ${\mathcal L}$ on $\C^n$ is given by
\Bea {\mathcal L}&=&\frac{1}{2}\,
\sum_{j=1}^n\left(Z_j\overline{Z}_j+\overline{Z}_jZ_j\right)
\Eea 
where 
$Z_j= \frac{\partial}{\partial z_j} + \frac{1}{2}\bar{z}_j,~
\overline{Z}_j= -\frac{\partial}{\partial \bar{z}_j} + \frac{1}{2}
z_j,~ j=1,2,\ldots n.$ Here $\frac{\partial}{\partial z_j}$ and
$\frac{\partial}{\partial \bar{z}_j}$  denote the complex derivatives 
$\frac{\partial}{\partial_{x_j}} \mp i \frac{\partial}{\partial_{y_j}}$ 
respectively.  Nonlinear Schr\"{o}dinger equation for the twisted Laplacian
has been studied in \cite{pkrvks,pkrvks1,ZZ}. 
For spectral theory of twisted Laplacian $\LL$ we refer to \cite{T,T1}
and for Schr\"{o}dinger equation
we refer to \cite{C,TT}.
Now we state the main theorem of this paper.

\begin{Theorem}\label{criticallocal} Let $f\in \tilde{W}_{\LL}^{1,2}(\C^n)$ and 
$G$ be as in (\ref{nlc1}) and (\ref{nlc2}) with $\alpha=\frac{2}{n-1}$ 
and $n\geq 2$. 
Initial value problem (\ref{cmainpde}), (\ref{cinidata}) has maximal 
solution $u\in C((T_*,T^*),\tilde{W}_{\LL}^{1,2})\cap L_{{\small\mbox{loc}}}^{q_1}\left((T_*,T^*), \tilde{W}_{\LL}^{1,p_1}(\C^n)\right)$ for every admissible pair $(q_1,p_1)$, where $t_0\in (T_*,T^*)$. Moreover the following properties hold:
\begin{description}
\item[(i)(Uniqueness)] Solution is unique in $C((T_*,T^*),\tilde{W}_{\LL}^{1,2}(\C^n))\cap L^{\gamma}((T_*,T^*),$ $\tilde{W}_{\LL}^{1,\rho})$.
%, where $\rho=\frac{2n^2}{n^2-n+1},  \gamma=\frac{2n}{n-1}$.
\item[(ii)(Blowup alternative)] If $T^*<\infty$ then $\Vert u\Vert_{L^{q}\left((t_0,T^*),\tilde{W}_{\LL}^{1,p}\right)}=\infty$ for every admissible pair $(q,p)$ with 
$2<p$ and
$\frac{1}{q}=n\left(\frac{1}{2}-{1\over p}\right)$.
Similar conclusion holds if $T_*>-\infty$.
\item[(iii)(Stability)] If
$f_j\to f$ in $\tilde{W}_{\LL}^{1,2}(\C^n)$ then $\Vert u-\tilde{u_j}\Vert_{L^{q}\left(I,\tilde{W}_{\LL}^{1,p}(\C^n)\right)}\to 0$ as $j\to \infty$ for every admissible pair $(q,p)$ and every interval $I$ with $\overline{I}\subset (T_*,T^*)$, where  $u,\tilde{u_j}$ are solutions corresponding to $f,f_j$ respectively.

\item[(iv)(Conservation of charge and energy)] If $\psi: \C^n\times [0,\infty)\to \R$ is real valued, then we have conservation of charge, i.e., 
$\Vert u(\cdot,t)\Vert_{L^2(\C^n)}=\Vert f\Vert_{L^2(\C^n)}$  and conservation of 
energy $E(u(\cdot,t))=E(u(\cdot,t_0))=E(f)$ for each  $t\in (T_*,T^*)$, where 
\bea\label{energy1}\hspace{1.4cm} E(f)&=&\frac{1}{4}\sum_{j=1}^{n}
\left(\Vert Z_jf\Vert_{L^2(\C^n)}^2+\Vert \bar{Z}_jf\Vert_{L^2(\C^n)}^2 \right) 
+\int_{\C^n}\tilde{G}(z,|f|)dz
\eea
and $\tilde{G}:\C^n\times [0,\infty)\to \R$ is given by 
\bea\label{energy2}\tilde{G}(z,\sigma)=\int_{s=0}^{\sigma}s\psi(z,s)ds=\int_{s=0}^{\sigma}G(z,s)ds.
\eea

\end{description}

\end{Theorem}

%\begin{Remark}\label{admiss}  $(\gamma,\rho)=(\frac{2n}{n-1},\frac{2n^2}{n^2-n+1})$ is an admissible pair according to definition 3.1 in \cite{pkrvks1}.
%\end{Remark}
To prove local existence, we truncate the given nonlinearity $G$ and obtain 
solutions for the truncated problem.
For $m\geq 1$, consider $G_m(z,u)=\psi_m(z,|u|)u: \C^n\times \C\to \C$, where
\bea \psi_m(z,\sigma)=\left\{\begin{array}{cc} \psi(z,\sigma) &\mbox{~if~} 
0\leq\sigma\leq m\\
m^2\left(\frac{\psi(z,\sigma)}{\sigma^2}-\frac{\psi(z,m)}{\sigma^2}+
\frac{\psi(z,m)}{m^2}\right) 
& \mbox{~if~}  \sigma\geq m \\
\end{array} \right. \eea
For $m=0$, we define $G_0(z,u)=G(z,u)$ and $\psi_0(z,|u|)=\psi(z,|u|)$.
Note that $\psi_m$ is differentiable at $\sigma=m$ with respect to $\sigma$ 
and also note that $G_m$ will satisfy (\ref{nlc1}) and (\ref{nlc2}) 
with $\alpha=\frac{2}{n-1}$ as well as $\alpha=0$. 
For $m\geq 1$, $G_m(z,\cdot):\C\to\C$ is globally Lipschitz from mean value theorem and 
\bea\label{lip} \left|G_m(z,u)-G_m(z,v)\right|\leq C_m|u-v| \mbox{~~for~~} m\geq 1
\eea
where constant $C_m$ depends on $m\in\Z_{\geq 1}$ but independent of $z\in\C^n$ and $u,v\in C$.
Moreover by mean value theorem we also see that 
\bea\label{lip1} \left|G_m(z,u)-G_m(z,v)\right|\leq C(|u|+|v|)^{\frac{2}{n-1}}|u-v| \mbox{~~for~~} m\geq 0
\eea
where constant $C$ is independent of $m\in \Z_{\geq 0},z\in \C^n$ and $u,v\in C$. 

Since $F_0$ satisfies estimate (\ref{nlc2})  with $\alpha=\frac{2}{n-1}$, therefore we conclude that
\bea\label{nlc3}
|F_m(z,\sigma)|\leq C\sigma^{\frac{2}{n-1}},
\eea
 where $F_m= \psi_m, \partial_{x_j} \psi_m $, 
$ \partial_{y_j}\psi_m$, $\sigma\partial_{\sigma}\psi_m(x,y,\sigma)$ 
with $1\leq j\leq n$ and constant $C$ is independent of 
$m$.

In view of Duhamel's formula (see, Lemma A.1 in \cite{pkrvks}) and in order to 
find solution for given IVP (\ref{cmainpde}), (\ref{cinidata}) 
with initial value $f\in \tilde{W}_{\LL}^{1,2}(\C^n)$,   
it is sufficient to find the solution of 
the following equation
\Bea \label{opr5} u(z,t)=e^{-i(t-t_0)\mathcal{L}}f (z)
-i\displaystyle\int_{t_0}^t e^{-i(t-s)\mathcal{L}}G(z,u(z,s))ds.
\Eea
In view of Banach fixed point theorem, 
for given $T>0$, $u:\C^n\times (t_0-T,t_0+T)\to \C$ and $m\geq 0$,
 we  consider the operator ${\mathcal H_m}$ given by the following 
\bea \label{opr} {\mathcal H_m}(u)(z,t)=e^{-i(t-t_0)\mathcal{L}}f (z)
-i\displaystyle\int_{t_0}^t e^{-i(t-s)\mathcal{L}}G_m(z,u(z,s))ds.
\eea

\section{Some auxilliary estimates}

\begin{Lemma}\label{embedding1} [Sobolev Embedding Theorem] 
We have the continuous inclusion 
\begin{equation*}
\begin{array}{lll}
\tilde{W}_{\LL}^{1,p_1}(\C^n)\hookrightarrow L^{p_2}(\C^n) & \text{ for } p_1\leq p_2 \leq
                                                     \frac{2np_1}{2n-p_1} &\text{ if
                                                     } p_1<2n   \\
                                                     &\text{ for } p_1 \leq
                                                     p_2 < \infty &\text{ if
                                                     }  p_1=2n  \\
                                                     &\text{ for } p_1 \leq
                                                     p_2 \leq \infty &\text{ if
                                                     }  p_1>2n.

\end{array}
\end{equation*}

\end{Lemma}

\begin{proof}
Let $f\in \tilde{W}_{\LL}^{1,p_1}(\C^n)$ and $\epsilon>0$. Consider $u_{\epsilon}=e^{-\epsilon\LL}f$. 
Then $u_{\epsilon}\in \tilde{W}_{\LL}^{1,p_1}(\C^n)\cap C^{\infty}(\C^n)$ and we have
\Bea
2|u_{\epsilon}| \frac{\partial}{\partial x_j}|u_{\epsilon}|=\frac{\partial}{\partial x_j}
( \overline{u_{\epsilon}} u_{\epsilon})=
2\Re \left( \overline{u_{\epsilon}} \,\frac{\partial}{\partial x_j}u_{\epsilon}\right)= 
2\Re\left(  \overline{u_{\epsilon}}( \frac{\partial}{\partial x_j}
-\frac{iy_j}{2} )\,u_{\epsilon}\right).
\Eea
Hence on the set $A=\set{z \in\C^n  \suchthat  u_{\epsilon}(z)\neq 0 },$ we have
$$ \left| \frac{\partial}{\partial x_j}|u_{\epsilon}| \right| =  
\left| \Re \left(  \frac{\overline{u_{\epsilon}}}{|u_{\epsilon}|} ( \frac{\partial}
{\partial x_j}-\frac{iy_j}{2} )\,u_{\epsilon} \right) \right|\leq \frac{1}{2}(|Z_ju_{\epsilon}|+|\bar{Z}_ju_{\epsilon}|) .$$ 
Similarly  $ \left|\frac{\partial}{\partial y_j}|u_{\epsilon}|
\right|\leq \frac{1}{2}(|Z_ju_{\epsilon}|+|\bar{Z}_ju_{\epsilon}|)$ on $A$.
Note that $\Vert u_{\epsilon}\Vert_{L^{p_2}(\C^n)} =\Vert u_{\epsilon}\chi_{A}\Vert_{L^{p_2}(\C^n)}$. 
By usual Sobolev embedding on $\C^n$ and above observations, we have 
inequality $\Vert u_{\epsilon}\Vert_{L^{p_2}(\C^n)}\leq 
C\Vert|u_{\epsilon}\chi_{A}|\Vert_{W^{1,p_1}}\leq C\,
\Vert u_{\epsilon} \Vert_{\tilde{W}_{\LL}^{1,p_1}}$. 
Since  $u_{\epsilon}=e^{-\epsilon\LL}f\to f$ in $\tilde{W}_{\LL}^{1,p_1}(\C^n)$ and 
also in $L^{p_2}(\C^n)$ as $\epsilon\to 0$ (see \cite{pkrvks1}), therefore
 we have $\Vert f\Vert_{L^{p_2}(\C^n)}\leq
C\Vert f \Vert_{\tilde{W}_{\LL}^{1,p_1}(\C^n)}$, where constant $C$ is a generic constant independent of $f$. Hence Lemma is proved.

\end{proof}

\begin{Lemma}\label{estimates2} Let $u,v\in L^{\gamma}\left(I,\tilde{W}_{\LL}^{1,\rho}(\C^n)\right)$ 
for some interval $I$, then following estimate holds for each $m\in\Z_{\geq 0}$
\bea\label{criticalest1}
\begin{split}\Vert G_m(z,u)-G_m(z,v)\Vert_{L^{\gamma'}\left(I,L^{\rho'}(\C^n)\right)} \leq  
C\Vert u-v\Vert_{L^{\gamma}\left(I,L^{\rho}(\C^n)\right)}\times\\
\left(\Vert u\Vert_{L^{\gamma}\left(I,\tilde{W}_{\LL}^{1,\rho}(\C^n)\right)}+
\Vert v\Vert_{L^{\gamma}\left(I,\tilde{W}_{\LL}^{1,\rho}(\C^n)\right)} \right)^{\frac{2}{n-1}}
\end{split}\eea
where constant $C$ is independent of $u,v,m, t_0$ and $I$.
\end{Lemma}

\begin{proof} Since $\frac{1}{\rho'}=\frac{1}{\rho}+
\frac{n-1}{n^2}=\frac{1}{\rho}+\frac{2}{n-1}\cdot\frac{n-1}{n\gamma}$, by using H\"{o}lder's 
inequality in the $z$-variable in 
(\ref{lip1}) and by embedding theorem (Lemma \ref{embedding1}), we get for each $t\in I$
\bea&&\Vert G_m(\cdot,u(\cdot,t))-G_m(\cdot,v(\cdot,t))\Vert_{L^{\rho'}(\C^n)} \nonumber\\
&\leq &C\Vert (u-v)(\cdot,t)\Vert_{L^{\rho}(\C^n)}\left(\Vert u(\cdot,t)
\Vert_{L^{\frac{n\gamma}{n-1}}(\C^n)}+\Vert v(\cdot,t)\Vert_{L^{\frac{n\gamma}{n-1}}(\C^n)}
\right)^{\frac{\gamma}{n}}\nonumber\\
&\leq &C\Vert (u-v)(\cdot,t)\Vert_{L^{\rho}(\C^n)}\left(\Vert 
u(\cdot,t)\Vert_{\tilde{W}_{\LL}^{1,\rho}(\C^n)}
+\Vert v(\cdot,t)\Vert_{\tilde{W}_{\LL}^{1,\rho}(\C^n)}\right)^{\frac{\gamma}{n}}.\eea
Since $\frac{1}{\gamma'}=\frac{1}{\gamma}+\frac{1}{n}$, by taking $L^{\gamma'}$ 
norm in $t$-variable in this inequality and then by using 
H\"{o}der's inequality we get desired estimate (\ref{criticalest1}). 
\end{proof}

%Now we state a Lemma, for proof we refer to Lemma 2.2 and Lemma 7.1 in \cite{pkrvks1}.

%\begin{Lemma}\label{continuity0} 
%Let $I$ be an interval and $u\in C(I,\tilde{W}_{\LL}^{1,2}(\C^n))$. Then $u\in 
%C(I,L^{\frac{2n}{n-1}}(\C^n))$ and $G_m(z,u(z,t))\in C(I, L^{\frac{2n}{n+1}}(\C^n))$ for $m\geq 0$.
%\end{Lemma}

\begin{Lemma}\label{estimates3} Let $I$ be a bounded interval and $u\in 
L^{\infty}(I,\tilde{W}_{\LL}^{1,2}(\C^n))\cap L^{\gamma}$ $(I,\tilde{W}_{\LL}^{1,\rho}(\C^n))$, 
then following estimate holds 
\Bea&\Vert G_m(z,u(z,t))-G(z,u(z,t))\Vert_{L^{\gamma'}\left(I,L^{\rho'}(\C^n)\right)}\\
&\leq C|I|^{\frac{n-1}{2n}}m^{-\frac{1}{n(n-1)}}\Vert u\Vert_{L^{\infty}
\left(I,\tilde{W}_{\LL}^{1,2}(\C^n)\right)}^{\frac{n^2-n+1}{n(n-1)}}\Vert 
u\Vert_{L^{\gamma}\left(I,\tilde{W}_{\LL}^{1,\rho}(\C^n)\right)}^{\frac{2}{n-1}}
\Eea
for all $m\geq 1$, where constant $C$ is independent of $m,u$ and $I$. 
\end{Lemma}

\begin{proof} Note that  
$$G_m(z,u(z,t))-G(z,u(z,t))=(u\chi_{|u(z,t)|>m}(z,t))(\psi_m(z,|u|)-\psi(z,|u|)).$$
Therefore 
$|G_m(z,u(z,t))-G(z,u(z,t))|\leq C|u\chi_{|u(z,t)|>m}(z,t)|~|u|^{\frac{2}{n-1}}$. 
By Taking 
$L^{\rho'}$-norm in the $z$-variable, we have
\bea\label{norminz}\Vert G_m(z,u)-G(z,u)\Vert_{L^{\rho'}(\C^n)}&\leq &C\Vert u\chi_{|u|>m}(\cdot,t)\Vert_{L^{\rho}(\C^n)}\Vert u(\cdot,t)\Vert_{L^{\frac{n\gamma}{n-1}}(\C^n)}^{\frac{\gamma}{n}} \nonumber\\
&\leq & C\Vert u\chi_{|u|>m}(\cdot,t)\Vert_{L^{\rho}(\C^n)}\Vert u(\cdot,t)\Vert_{\tilde{W}_{\LL}^{1,\rho}(\C^n)}^{\frac{\gamma}{n}}.  
\eea
Now we observe the following
\Bea\Vert u\chi_{|u|>m}(\cdot,t)\Vert_{L^{\rho}(\C^n)}^{\rho}&=&\int_{\C^n} |u|^{\rho}\chi_{|u|>m}(z,t)dz\\
&\leq & \int_{\C^n} m^{-\frac{\rho}{n(n-1)}}|u|^{\frac{2n}{n-1}}dz \\
&\leq & m^{-\frac{\rho}{n(n-1)}}
\Vert u\Vert_{L^{\frac{2n}{n-1}}(\C^n)}^{\frac{(n^2-n+1)\rho}{n(n-1)}}\\
&\leq & m^{-\frac{\rho}{n(n-1)}}
\Vert u\Vert_{\tilde{W}_{\LL}^{1,2}(\C^n)}^{\frac{(n^2-n+1)\rho}{n(n-1)}}\\
\Vert u\chi_{|u|>m}(\cdot,t)\Vert_{L^{\rho}} &\leq & m^{-\frac{1}{n(n-1)}}
\Vert u\Vert_{\tilde{W}_{\LL}^{1,2}(\C^n)}^{\frac{(n^2-n+1)}{n(n-1)}}. \Eea
By taking $L^{\gamma}$-norm in the $t$-variable we have 
\bea \label{normint}\Vert u\chi_{|u|>m}\Vert_{L^{\gamma}\left(I,L^{\rho}(\C^n)\right)} &\leq & |I|^{\frac{n-1}{2n}}m^{-\frac{1}{n(n-1)}}
\Vert u\Vert_{L^{\infty}\left(I,\tilde{W}_{\LL}^{1,2}(\C^n)\right)}^{\frac{(n^2-n+1)}{n(n-1)}}.
\eea
By taking $L^{\gamma'}$-norm in the $t$-variable in estimate (\ref{norminz}) and using H\"{o}lder's inequality, we get 
\Bea \Vert G_m(z,u)-G(z,u)\Vert_{L^{\gamma'}\left(I,L^{\rho'}\right)} &\leq & C\Vert u\chi_{|u|>m}\Vert_{L^{\gamma}\left(I,L^{\rho}\right)} \Vert u\Vert_{L^{\gamma}\left(I,\tilde{W}_{\LL}^{1,\rho}\right)}^{\frac{2}{n-1}}.
\Eea
By using inequality (\ref{normint}) in the above inequality, we get the
 desired estimate.

\end{proof}

\begin{Lemma}\label{estimates4} Let $u\in L^{\gamma}\left(I,\tilde{W}_{\LL}^{1,\rho}(\C^n)\right)$ for some interval $I$. Then for each $m\in\Z_{\geq 0}$, $G_m(z,u(z,t))\in L^{\gamma'}\left(I,\tilde{W}_{\LL}^{1,\rho'}(\C^n)\right)$ and following estimates hold:
\bea\label{estimates41}\Vert SG_m(z,u(z,t))\Vert_{L^{\gamma'}\left(I,L^{\rho'}(\C^n)\right)} \leq  C\Vert u\Vert_{L^{\gamma}\left(I,\tilde{W}_{\LL}^{1,\rho}(\C^n)\right)}^{\frac{n+1}{n-1}}\\
\label{estimates42}\Vert G_m(z,u(z,t))\Vert_{L^{\gamma'}\left(I,\tilde{W}_{\LL}^{1,\rho'}(\C^n)\right)} \leq  C\Vert u\Vert_{L^{\gamma}\left(I,\tilde{W}_{\LL}^{1,\rho}(\C^n)\right)}^{\frac{n+1}{n-1}}
\eea
where $S=Z_j,\overline{Z}_j$ or $Id$, $1\leq j \leq n$ and constant $C$ is 
independent of $u$ and $I$.
\end{Lemma}

\begin{proof} Since $\psi_m,\partial_{x_j}\psi_m,\partial_{y_j}\psi_m,|u|\partial_{|u|}\psi_m$ satisfy estimate (\ref{nlc3}), therefore  we have 
$$ |SG_m(z,u)|\leq C|u|^{\frac{2}{n-1}}(|u|+|Z_ju|+|\bar{Z}_ju|)$$
where $S=Z_j,\overline{Z}_j$ ($1\leq j \leq n$) or $Id$,
see Lemma 3.4 in \cite {pkrvks1}.
Now estimate (\ref{estimates41}) follows from H\"{o}lder's inequality and embedding theorem (Lemma \ref{embedding1}) 
as we used in the proof of Lemma \ref{estimates2}. Estimate (\ref{estimates42}) is a consequence of estimate (\ref{estimates41}).
\end{proof}

\begin{Proposition}\label{convergenceinmixed0} Let $I$ be a bounded interval 
such that $t_0\in\overline{I}$. 
\begin{description}
\item[(i)] If $u,v\in L^{\gamma}\left(I,\tilde{W}_{\LL}^{1,\rho}(\C^n)\right)$, then
$\mathcal{H}_mu-\mathcal{H}_mv\in L^{q}\left(I,L^p(\C^n)\right)$ for every admissible pair $(q,p)$, for every $m\in\Z_{\geq 0}$ and following estimate holds:
\bea\label{convergenceinmixed01}
&&\Vert\mathcal{H}_mu-\mathcal{H}_mv\Vert_{L^q\left(I,L^p(\C^n)\right)}\\
&\leq &C\Vert u-v\Vert_{L^{\gamma}\left(I,L^{\rho}(\C^n)\right)} \left(\Vert u\Vert_{L^{\gamma}\left(I,\tilde{W}_{\LL}^{1,\rho}(\C^n)\right)}+\Vert v\Vert_{L^{\gamma}\left(I,\tilde{W}_{\LL}^{1,\rho}(\C^n)\right)}\right)^{\frac{2}{n-1}}.\nonumber
\eea
\item[(ii)] If $u\in L^{\infty}\left(I,\tilde{W}_{\LL}^{1,2}(\C^n)\right)\cap L^{\gamma}\left(I,\tilde{W}_{\LL}^{1,\rho}(\C^n)\right)$, then $\mathcal{H}_mu-\mathcal{H}u\in L^{q}\left(I,L^p(\C^n)\right)$ for every admissible pair $(q,p)$, for every $m\in\Z_{\geq 1}$ and following estimate holds
\bea\label{convergenceinmixed02}
&&\Vert \mathcal{H}_mu-\mathcal{H}u\Vert_{L^q\left(I,L^p(\C^n)\right)}\\
&\leq &C|I|^{\frac{n-1}{2n}}m^{-\frac{1}{n(n-1)}}\Vert u\Vert_{L^{\infty}\left(I,\tilde{W}_{\LL}^{1,2}(\C^n)\right)}^
{\frac{n^2-n+1}{n(n-1)}}
\Vert u\Vert_{L^{\gamma}\left(I,\tilde{W}_{\LL}^{1,\rho}(\C^n)\right)}^{\frac{2}{n-1}}\nonumber        \eea
\end{description}
where constant $C$ is independent of $u,v,m$ and $t_0$.
\end{Proposition}
\begin{proof} Estimate (\ref {convergenceinmixed01}) follows from Strichartz estimates (Theorem 3.3 in \cite{pkrvks1}) and Lemma \ref{estimates2}, whereas estimate (\ref {convergenceinmixed02}) follows from Strichartz estimates and Lemma \ref{estimates3}. 
\end{proof}

Now we state the following Proposition, which is useful in proving stability. 
%In view of Lemma \ref{estimates2}, for proof, we refer to Proposition 4.3 in \cite{pkrvks}.

\begin{Proposition}\label{forstability1}
 Let $\Phi$ be a continuous complex valued function on $\C$ such that $|\Phi(w)|\leq C|w|^{\frac{2}{n-1}}$ with $n\geq 2$. Let
 $\{ u_m\} $ be a bounded sequence in 
 $L^{\gamma}\left(I,\tilde{W}_{\LL}^{1,\rho}(\C^n)\right)$ for some interval $I$.
If $u_m\rightarrow u$ in $L^{\gamma}(I,L^{\rho}(\C^n))$  then $u\in L^{\gamma}\left(I,\tilde{W}_{\LL}^{1,\rho}(\C^n)\right)$ and
$[\Phi(u_m)-\Phi(u)]Su \rightarrow 0$ in $L^{\gamma'}\left(I,L^{\rho'}(\C^n) \right),$ 
for  $S=\mbox{Id},Z_j,\overline{Z}_j;1\leq j\leq n$.
\end{Proposition}

\begin{proof}
First we will prove $u\in L^{\gamma}\left(I,\tilde{W}_{\LL}^{1,\rho}(\C^n)\right)$.
By duality argument (also see Lemma A.2.1 in \cite{GV}), we have
\bea \Vert Su\Vert_{L^{\gamma}(I,L^{\rho}(\C^n))}&=& \sup_{\phi\in 
C_c^{\infty}(\C^n\times I),\Vert \phi\Vert_{L^{\gamma'}(I,L^{\rho'}(\C^n))}\leq 1}
\left|\left<Su,\phi \right>_{z,t}\right|\nonumber\\
&=&\sup_{\phi}\left|\left<u,S^*\phi \right>_{z,t}\right|\nonumber\\  
&=&\sup_{\phi} \lim_{m\to \infty} \left|\left<u_m,S^*\phi \right>_{z,t}\right|\nonumber\\
&=&\sup_{\phi}\lim_{m\to \infty} \left|\left<Su_m,\phi \right>_{z,t}\right|\nonumber\\
&\leq &\sup_{\phi}\liminf_{m\to \infty} \Vert Su_m\Vert_{L^{\gamma}(I,L^{\rho}(\C^n))}
\Vert \phi\Vert_{L^{\gamma'}(I,L^{\rho'}(\C^n))}\nonumber\\
\label{duality}&\leq &\liminf_{m\to \infty} \Vert Su_m\Vert_{L^{\gamma}(I,L^{\rho}(\C^n))}
\eea
for $S=Z_j,\bar{Z}_j$; $1\leq j \leq n$. Therefore
\Bea \Vert u\Vert_{L^{\gamma}(I,\tilde{W}_{\LL}^{1,\rho}(\C^n))}\leq \liminf_{m\to \infty}
\Vert u_m\Vert_{L^{\gamma}(I,\tilde{W}_{\LL}^{1,\rho}(\C^n))}<\infty.
\Eea

Since $u_m\to u$ in $L^{\gamma}(I,L^{\rho}(\C^n))$, we can extract a subsequence 
still denoted by ${u_k}$ such that 
$$ \Vert u_{k+1}-u_{k}\Vert_{L^{\gamma}(I,L^{\rho}(\C^n))}\leq \frac{1}{2^k}$$ 
for all $k\geq 1$ and $u_k(z,t)\to u(z,t)$ a.e. Hence by continuity of $\Phi$,
\bea [\Phi(u_k)-\Phi(u)]Su \rightarrow 0 \hskip.1in~\mbox{for a.e} ~(z,t)\in \C^n \times I .\eea
We establish the norm convergence by appealing to a dominated convergence argument in 
$z$ and $t$ variables successively.

Consider the function  $H(z,t)=\sum_{k=1}^\infty |u_{k+1}(z,t)-u_k(z,t)|$. Clearly
  $H \in L^{\gamma}(I,L^{\rho}(\C^n))$. Also for $l>k$, $ |(u_l-u_k)(z,t)|\leq |u_l-u_{l-1}|+\cdots +|u_{k+1}-u_k|\leq H(z,t)$,
 hence $ |u_k -u| \leq H.$
This leads to the pointwise almost everywhere inequality \Bea \label{dominate} |u_k(z,t) | \leq |u(z,t)| + H(z,t)=v(z,t).\Eea
Hence 
\bea\label{cri4.5} |\left[\Phi(u_k)-\Phi(u)\right]Su(z,t)|^{\rho'} \leq C [v^{\frac{2}{n-1}}+|u|^{\frac{2}{n-1}} ]^{\rho'} ~~ |Su(z,t)|^{\rho'} .
\eea
Since $u, v \in L^{\gamma}(I,L^{\rho}(\C^n))$,  using H\"{o}lder's inequality with $\frac{1}{\rho'}=\frac{1}{\rho}+
\frac{n-1}{n^2}=\frac{1}{\rho}+\frac{2}{n-1}\cdot\frac{n-1}{n\gamma}$ and Lemma \ref{embedding1}, 
 we get
\bea \label{cri4.6}
 &&\int_{\C^n}  [v^{\frac{2}{n-1}}+|u|^{\frac{2}{n-1}} ]^{\rho'} ~~ |Su(z,t)|^{\rho'}dz \\
 && \leq 
( \| v(\cdot,t) \|_{L^{\frac{n\gamma}{n-1}}(\C^n)} +   \|u(\cdot,t) \|_{L^{\frac{n\gamma}{n-1}}(\C^n)}  )^ {\frac{\rho'\gamma}{n}} \| Su(\cdot,t)\|_{L^{\rho}(\C^n)} ^{\rho'}.\nonumber\\
&& \leq 
( \| v(\cdot,t) \|_{\tilde{W}_{\LL}^{1,\rho}(\C^n)} +   \|u(\cdot,t) \|_{\tilde{W}_{\LL}^{1,\rho}(\C^n)}  )^ {\frac{\rho'\gamma}{n}} \| Su(\cdot,t)\|_{L^{\rho}(\C^n)} ^{\rho'}<\infty\nonumber
 \eea
for a.e. $t\in I$.
Thus in view of (\ref{cri4.5}), (\ref{cri4.6}) and using dominated convergence theorem in the $z$-variable, 
we see that  
 \bea \label{cri4.8}\Vert \left[\Phi(u_k)-\Phi(u)\right]Su(\cdot,t)\Vert_{L^{p'}(\C^n)}
\rightarrow 0\eea
 as $k\rightarrow \infty,$ for a.e. $t$.

Again, in view of  (\ref{cri4.5}) and (\ref{cri4.6}), we get
\Bea
 &&\| [\Phi(u_k)-\Phi(u)]\, Su(\cdot,t) \|_{L^{\rho'}(\C^n)}\\ 
 &&\leq C( \| v(\cdot,t) \|_{\tilde{W}_{\LL}^{1,\rho}(\C^n)} +   \|u(\cdot,t) \|_{\tilde{W}_{\LL}^{1,\rho}(\C^n)}  )^ {\frac{\gamma}{n}} \| Su(\cdot,t)\|_{L^{\rho}(\C^n)}.
\Eea
Since $\frac{1}{\gamma'}=\frac{1}{\gamma}+\frac{1}{n}$, an application of the H\"{o}lder's inequality in the $t$-variable shows that 
\Bea &&\| [\Phi(u_k)-\Phi(u)]\, Su \|_{L^{\gamma'}(I,L^{\rho'}(\C^n))}\\ 
 &&\leq C(\| v \|_{L^{\gamma}(I,\tilde{W}_{\LL}^{1,\rho}(\C^n))} +   \|u \|_{L^{\gamma}(I,\tilde{W}_{\LL}^{1,\rho}(\C^n))}  )^ {\frac{\gamma}{n}} \| Su\|_{L^{\gamma}(I,L^{\rho}(\C^n))}.
\Eea 
Hence a further application of dominated convergence theorem with
(\ref{cri4.8}) shows  that $\Vert \left(\Phi(u_k)-\Phi(u)\right)Su\Vert_{L^{\gamma'}(I,L^{\rho'})}\rightarrow 0$, as
$k\rightarrow \infty$. 

Thus we have shown that $\left[\Phi(u_{m_k})-\Phi(u)\right]Su\rightarrow 0$
in $L^{\gamma'}(I,L^{\rho'}(\C^n))$ for some subsequence $u_{m_k}$
whenever  $u_m \to u$ in $L^{\gamma}(I,L^{\rho}(\C^n))$. But the above arguments are also valid if we had started with any subsequence of $u_m$. 
It follows that any subsequence of $\left[\Phi(u_{m})-\Phi(u)\right]Su$ has a  subsequence that converges to $0$
in $L^{\gamma'}(I,L^{\rho'}(\C^n))$. From this we conclude that the original sequence 
$\left[\Phi(u_m)-\Phi(u)\right]Su$ converges to zero in $L^{\gamma'}(I,L^{\rho'}(\C^n))$, hence the proposition.
  \end{proof}

\section{Proof of Theorem \ref{criticallocal}}

\begin{proof}{(of Theorem \ref{criticallocal}): (Local existence):}
Since $G_m(z,\cdot):\C\to\C$ is globally lipschitz for each $m\geq 1$, 
see (\ref{lip}), therefore from
 \cite{pkrvks1} it follows that there exists a unique global solution 
$u_m\in C(\R,\tilde{W}_{\LL}^{1,2}(\C^n))$ of the initial value problem
\bea \label{cmainpdem}i\partial_tv(z,t)\!\! -\!\!\!\LL
v(z,t)&=&G_m(z, v), ~~~~~~~~~~~~z \in \C^n, ~t \in \R \\
\label{cinidatam} v(\cdot,t_0)&=& f. 
\eea
Furthermore $\mathcal{H}_mu_m=u_m$ and  $u_m\in L_{\small\mbox{loc}}^q(\R,\tilde{W}_{\LL}^{1,p}(\C^n))$ 
for every admissible pair $(q,p)$. 
We deduce from Lemma \ref{estimates4} and Strichartz estimates (Theorem 3.3 in \cite{pkrvks1}) that
\bea\label{local01}
\begin{split}&\Vert u_m\Vert_{L^q\left((t_0,t_0+T),\tilde{W}_{\LL}^{1,p}(\C^n)\right)}\\
&\leq \Vert e^{-i(t-t_0)\LL}f\Vert_{L^q\left((t_0,t_0+T),\tilde{W}_{\LL}^{1,p}(\C^n)\right)}
+C\Vert u_m\Vert_{L^{\gamma}\left((t_0,t_0+T),\tilde{W}_{\LL}^{1,\rho}(\C^n)\right)}^{\frac{n+1}{n-1}}.\end{split}
\eea
Let $l\geq m$, we see that
\Bea  u_m-u_l=(\mathcal{H}_m(u_m)-\mathcal{H}_m(u_l))
+(\mathcal{H}_m(u_l)-\mathcal{H}(u_l))+(\mathcal{H}(u_l)-\mathcal{H}_l(u_l)).\Eea
From Proposition \ref{convergenceinmixed0}, we deduce that
\bea\label{local3} 
\hspace{-.2cm}\Vert u_m-u_l\Vert_{L^q((t_0,t_0+T),L^p(\C^n))}
\leq C\left(\Vert u_m\Vert_{L^{\gamma}((t_0,t_0+T),\tilde{W}_{\LL}^{1,\rho})}
+\Vert u_l\Vert_{L^{\gamma}((t_0,t_0+T),\tilde{W}_{\LL}^{1,\rho})}\right)^{\frac{2}{n-1}}\times\nonumber\\
\left(\Vert u_m-u_l\Vert_{L^{\gamma}((t_0,t_0+T),L^{\rho})}
+T^{\frac{n-1}{2n}}m^{-\frac{1}{n(n-1)}}\Vert u_l\Vert_{L^{\infty}((t_0,t_0+T),\tilde{W}_{\LL}^{1,2})}^{\frac{n^2-n+1}{n(n-1)}}\right).
\eea
We choose $T\leq\pi$, therefore we can take constant $C$ to be independent of $T$.
Let $\tilde{C}$ be larger than the constant $C$ that appear in (\ref{local01}),
 (\ref{local3}), (\ref{convergenceinmixed01}), (\ref{convergenceinmixed02}) 
and in Strichartz estimates for the particular choice of the admissible pair $(q,p)=(\gamma,\rho)$.
Fixed $\delta$ small enough so that
\bea\label{delta}\tilde{C}(4\delta)^{\frac{2}{n-1}} <\frac{1}{2}. \eea
We claim that if $0<T\leq \pi$ is such that 
\bea\label{delta1}\Vert e^{-i(t-t_0)\LL}f\Vert_{L^{\gamma}\left((t_0,t_0+T),\tilde{W}_{\LL}^{1,\rho}(\C^n)\right)}\leq \delta \eea
then
\bea\label{claim1} \sup_{m\geq 1} \Vert u_m\Vert_{L^{\gamma}\left((t_0,t_0+T),\tilde{W}_{\LL}^{1,\rho}(\C^n)\right)}&\leq &2\delta\\
\label{claim2}\sup_{m\geq 1}\Vert u_m\Vert_{L^q\left((t_0,t_0+T),\tilde{W}_{\LL}^{1,p}(\C^n)\right)}&<&\infty
\eea 
for every admissible pair $(q,p)$.
Let $\theta_m(t)=
\Vert u_m\Vert_{L^{\gamma}\left((t_0,t_0+t),\tilde{W}_{\LL}^{1,\rho}(\C^n)\right)}$. 
From (\ref{local01}), we see that
\Bea \theta_m(t)\leq \delta+\tilde{C}\theta_m(t)^{\frac{n+1}{n-1}}.
\Eea
 If $\theta_m(t)=2\delta$ for some $t\in (t_0,t_0+T]$, then 
\Bea 2\delta\leq \delta+\tilde{C}(2\delta)^{\frac{n+1}{n-1}}<2\delta
\Eea
which is a contradiction. Since $\theta_m$ is a continuous function 
with $\theta_m(t_0)=0$ therefore we conclude that $\theta_m(t)<2\delta$ 
for all $t\in (t_0,t_0+T]$, which proves (\ref{claim1}). From (\ref{local01}), we see that
\Bea \sup_{m}\Vert u_m\Vert_{L^q\left((t_0,t_0+T),\tilde{W}_{\LL}^{1,p}(\C^n)\right)}
&\leq &\Vert e^{-i(t-t_0)\LL}f\Vert_{L^q\left((t_0,t_0+T),\tilde{W}_{\LL}^{1,p}(\C^n)\right)}+C(2\delta)^{\frac{n+1}{n-1}}\\
&\leq &C(q,p,n,\delta,f)<\infty.
\Eea
This proves (\ref{claim2}). Put $(q,p)=(\gamma,\rho)$ in (\ref{local3}), we see that 
\Bea \Vert u_m-u_l\Vert_{L^{\gamma}\left((t_0,t_0+T),L^{\rho}(\C^n)\right)}&\leq 
&\frac{1}{2}\left(\Vert u_m-u_l\Vert_{L^{\gamma}\left((t_0,t_0+T),L^{\rho}(\C^n)\right)}+CT^{\frac{n-1}{2n}}m^{-\frac{1}{n(n-1)}}\right)\\
&\leq & 2CT^{\frac{n-1}{2n}}m^{-\frac{1}{n(n-1)}}\to 0 \mbox{~~as~~}m\to\infty.
\Eea
This shows that $u_m$ is a cauchy sequence in $L^{\gamma}\left((t_0,t_0+T),L^{\rho}(\C^n)\right)$ 
and from (\ref{local3}) it is also cauchy in $L^{q}\left((t_0,t_0+T),L^{p}(\C^n)\right)$ 
for every admissible pair $(q,p)$. Let $u$ be its limit, then $u_m\to u$ in 
$L^{q}\left((t_0,t_0+T),L^{p}(\C^n)\right)$ for every admissible pair  $(q,p)$. 
By duality 
argument (see (\ref{duality})) and from estimates (\ref{claim1}), (\ref{claim2}), we have
%\bea \Vert Su\Vert_{L^{\gamma}((t_0,t_0+T),L^{\rho}(\C^n))}&=& \sup_{\phi\in C_c^{\infty}(\C^n\times (t_0,t_0+T)),\Vert \phi\Vert_{L^{\gamma'}((t_0,t_0+T),L^{\rho'}(\C^n))}\leq 1} \left|\left<Su,\phi \right>_{z,t}\right|\nonumber\\&=&\sup_{\phi}\left|\left<u,S^*\phi \right>_{z,t}\right|\nonumber\\ &=&\sup_{\phi} \lim_{m\to \infty} \left|\left<u_m,S^*\phi \right>_{z,t}\right|\nonumber\\&=&\sup_{\phi}\lim_{m\to \infty} \left|\left<Su_m,\phi \right>_{z,t}\right|\nonumber\\&\leq &\sup_{\phi}\liminf_{m\to \infty} \Vert Su_m\Vert_{L^{\gamma}((t_0,t_0+T),L^{\rho}(\C^n))}\Vert \phi\Vert_{L^{\gamma'}((t_0,t_0+T),L^{\rho'}(\C^n))}\nonumber\\
%&\leq &\liminf_{m\to \infty} \Vert Su_m\Vert_{L^{\gamma}((t_0,t_0+T),L^{\rho}(\C^n))}\eea for $S=Z_j,\bar{Z}_j$; $1\leq j \leq n$.
%Therefore from estimates (\ref{claim1}), (\ref{claim2}), we have
\bea\label{uclaim1} \Vert u\Vert_{L^{\gamma}\left((t_0,t_0+T),
\tilde{W}_{\LL}^{1,\rho}(\C^n)\right)}&\leq &2\delta\\
\label{uclaim2}\Vert u\Vert_{L^q\left((t_0,t_0+T),\tilde{W}_{\LL}^{1,p}(\C^n)\right)}&<&\infty.
\eea 
From Lemma \ref{estimates4}, $G_m(z,u(z,t))\in L^{\gamma'}
\left((t_0,t_0+T),\tilde{W}_{\LL}^{1,\rho'}(\C^n)\right)$ 
for each $m\geq 0$. From Strichartz estimates
(Theorem 3.3 in \cite{pkrvks1}) and (\ref{opr}), 
$\mathcal{H}u\in L^q((t_0,t_0+T),\tilde{W}_{\LL}^{1,p}(\C^n))$  for every admissible pair $(q,p)$.

From Lemma \ref{estimates2} $\Vert G_m(z,u_m)-G_m(z,u)\Vert_{L^{\gamma'}\left((t_0,t_0+T),L^{\rho'}(\C^n)\right)}\to 0$ 
and from Lemma \ref{estimates3}, 
$\Vert G_m(z,u)-G(z,u)\Vert_{L^{\gamma'}\left((t_0,t_0+T),L^{\rho'}(\C^n)\right)}\to 0$ as $m\to \infty$. 
Therefore 
\Bea\Vert G_m(z,u_m)-G(z,u)\Vert_{L^{\gamma'}((t_0,t_0+T),L^{\rho'})}\to 
0~~\mbox{~~as~~}~~ m\to \infty.\Eea
 Since $u_m=\mathcal{H}_mu_m$ for each $m\geq 1$, from 
Strichartz estimates we deduce that 
\Bea \Vert u_m-\mathcal{H}u\Vert_{L^q\left((t_0,t_0+T),
L^p(\C^n)\right)}&=&\Vert\mathcal{H}_mu_m-\mathcal{H}u\Vert_{L^q\left((t_0,t_0+T),L^p(\C^n)\right)}\\
&\leq & C\Vert G_m(z,u_m)-G(z,u)\Vert_{L^{\gamma'}((t_0,t_0+T),L^{\rho'})}\to 0
\Eea
as $m\to\infty$. Therefore for $t\in (t_0,t_0+T)$
\bea\label{oprnorm1} u=\mathcal{H}u=e^{-i(t-t_0)\mathcal{L}}f (z)
-i\displaystyle\int_{t_0}^t e^{-i(t-s)\mathcal{L}}G(z,u(z,s))ds.\eea
From Strichartz estimates and estimate (\ref{uclaim2}), $u\in C([t_0,t_0+T],
\tilde{W}_{\LL}^{1,2})\cap L^q((t_0,t_0+T),\tilde{W}_{\LL}^{1,p}(\C^n))$ 
for every admissible pair $(q,p)$. 
In view of Lemma A.1 in \cite{pkrvks}, $u$ is also a solution to the initial value problem 
(\ref{cmainpde}), (\ref{cinidata}).
Similarly solution exists on the 
interval $[t_0-T',t_0]$ for some $T'>0$. Now we continue this process with initial time 
$t_0+T$ and $t_0-T'$. In this way we construct maximal solution 
$u\in C((T_*,T^*),\tilde{W}_{\LL}^{1,2})\cap  L_{{\small\mbox{loc}}}^{q}
\left((T_*,T^*), \tilde{W}_{\LL}^{1,p}(\C^n)\right)$ for every admissible pair $(q,p)$.

{\bf Conservation of charge and energy:}
From \cite{pkrvks1}, we have following conservation laws
\bea
\label{conserv1}\Vert u_m(\cdot,t)\Vert_{L^2(\C^n)}=\Vert f\Vert_{L^2(\C^n)}, 
~~t\in  \R\\
\label{conserv2}E_m(u_m(\cdot,t))=E_m(f), ~~t\in \R
\eea   
where
\bea
\hspace{.5cm}E_m(f)&=&\frac{1}{4}\sum_{j=1}^{n}
\left(\Vert Z_jf\Vert_{L^2(\C^n)}^2+\Vert \bar{Z}_jf\Vert_{L^2(\C^n)}^2 \right) 
+\int_{\C^n}\tilde{G_m}(z,|f(z)|)dz\\
\hspace{.5cm}\tilde{G_m}(z,\sigma)&=&\int_{s=0}^{\sigma}G_m(z,s)ds, ~~\sigma>0 .\nonumber
\eea
Since $u_m\to u$ in $L^{\infty}([t_0-T,t_0+T],L^2(\C^n))$ for sufficiently small $T>0$, 
therefore by taking limit $m\to\infty$ in (\ref{conserv1}), we get 
\Bea \Vert u(\cdot,t)\Vert_{L^2(\C^n)}=\Vert f\Vert_{L^2(\C^n)},~~t\in [t_0-T,t_0+T].
\Eea
By repeating this argument for any point in $(T_*,T^*)$, instead of $t_0$, 
we get conservation of charge on $(T_*,T^*)$.

Now we will prove conservation of energy. 
From (\ref{claim2}), for each $t\in (t_0-T,t_0+T)$, sequence 
$\Vert u_m(\cdot,t)\Vert_{\tilde{W}_{\LL}^{1,2}(\C^n)} $ is uniformly bounded and 
$u_m(\cdot,t)\to u(\cdot,t)$ in $L^2(\C^n)$, therefore by duality argument 
(see (\ref{duality})), we have
\Bea
\sum_{j=1}^{n}
(\Vert Z_ju(\cdot,t)\Vert_{L^2}^2+\Vert \bar{Z}_ju(\cdot,t)\Vert_{L^2}^2 ) 
\leq \liminf_{m\to\infty} \sum_{j=1}^{n}
(\Vert Z_ju_m(\cdot,t)\Vert_{L^2}^2+\Vert \bar{Z}_ju_m(\cdot,t)\Vert_{L^2}^2 ).
\Eea
Since $|\tilde{G_m}(z,|f(z)|)|\leq C |f|^{\frac{2n}{n-1}}$, 
$f\in \tilde{W}_{\LL}^{1,2}(\C^n)\subset L^{\frac{2n}{n-1}}(\C^n)$, therefore by dominated convergence theorem, 
$\int_{\C^n}\tilde{G_m}(z,|f(z)|)dz\to \int_{\C^n}\tilde{G}(z,|f(z)|)dz
$ as $m\to \infty$.
Since $u_m\to u$ in $L^q((t_0-T,t_0+T),L^p(\C^n))$ for every admissible pair, 
therefore after choosing a suitable subsequence, we have pointwise convergence 
$u_m(z,t)\to u(z,t)$ for a.e. $(z,t)$, see Proposition \ref{forstability1}
and Theorem 4.9 in Brezis \cite {BR}. Now we observe that
\Bea
\tilde{G}_m(z,|u_m(z,t)|)-\tilde{G}(z,|u_m(z,t)|)=
\int_{0}^{|u_m|}(G_m(z,\sigma)-G(z,\sigma))d\sigma\to 0
\Eea
pointwise as $m\to \infty$. Similarly $\tilde{G}(z,|u_m(z,t)|)-\tilde{G}(z,|u(z,t)|)\to 0
$ pointwise as $m\to \infty$, therefore $\tilde{G_m}(z,|u_m(z,t)|)-\tilde{G}(z,|u(z,t)|)
 \to 0$ pointwise as $m\to \infty$. By Fatou's Lemma, we have
\Bea \int_{\C^n}\tilde{G}(z,|u(z,t)|)dz\leq \liminf \int_{\C^n}\tilde{G_m}
(z,|u_m(z,t)|)dz.
\Eea
By above observations and energy conservation $E_m(u_m(\cdot,t))=E_m(f)$, 
we have
$E(u(\cdot,t))\leq \liminf E_m(u_m(\cdot,t))=\lim_{m\to \infty}E_m(f)=E(f)$ for $t\in (t_0-T,t_0+T)$. 
This shows that $t\to E(u(\cdot,t))$ has local maximum at $t_0$. 
But we can repeat this argument with any point in $(T_*,T^*)$, therefore
$t\to E(u(\cdot,t))$ has local maximum at every point of $(T_*,T^*)$. 
Since  $t\to E(u(\cdot,t))$ is continuous, therefore $E(u(\cdot,t))=E(f)$ for every $t\in (T_*,T^*)$.

\noindent{\bf Uniqueness:}
Uniqueness in $C((T_*,T^*),\tilde{W}_{\LL}^{1,2}(\C^n))\cap 
L_{\small\mbox{loc}}^{\gamma}\left((T_*,T^*),\tilde{W}_{\LL}^{1,\rho}(\C^n)\right)$ 
will follow from estimate (\ref{convergenceinmixed01}) with $m=0$
in Proposition \ref{convergenceinmixed0},
see uniqueness in \cite{pkrvks}.

\noindent{\bf Blowup alternative:} We prove blowup alternative by method of contradiction.
Let us assume that $T^*<\infty$ and $u\in L^{q}((t_0,T^*),\tilde{W}_{\LL}^{1,p})$ 
for some admissible pair $(q,p)$ with $2<p$ and 
$\frac{1}{q}=n\left(\frac{1}{2}-{1\over
p}\right)$.
Since $2< p<  \frac{2n}{n-1}$, $n\geq 2$, so $p<2n$.
 We choose admissible pair $(q_1,p_1)$ as follows
\Bea \frac{1}{p'_1}=\frac{1}{p_1}+\frac{2}{n-1}\left(\frac{1}{p}-\frac{1}{2n}\right),~~
  \frac{1}{q'_1}=\frac{1}{q_1}+\frac{2}{n-1}\frac{1}{q}.
\Eea
 Let us choose $s$ and $t$ such that $t_0\leq s <t<T^*$. Since 
$|S_jG(z,u(z,t))|\leq C|u|^{\frac{2}{n-1}}(|u|+|Z_ju|+|\overline{Z}_ju|)$ for $S_j=Id, Z_j,\overline{Z}_j $ ($1\leq j \leq n$) 
(see Lemma 3.4 in \cite{pkrvks1}),  
therefore by Lemma \ref{embedding1} and H\"{o}lder's inequality we see that 
\bea \label{blowup1}
\hspace{1cm}\Vert G(z,u(z,\tau))\Vert_{L^{q_1'}((s,t),\tilde{W}_{\LL}^{1,p_1'})}
\leq C\Vert u\Vert_{L^{q_1}((s,t),\tilde{W}_{\LL}^{1,p_1})}\Vert u\Vert_{L^{q}((s,t),\tilde{W}_{\LL}^{1,p})}^{\frac{2}{n-1}}.
\eea
 Since $(t_0,T^*)$ is a bounded interval, so we can choose constant 
$C$ independent of $s$ and $t$, where $t_0\leq s<t<T^*$.
Now we see that  
$$ u(z,\tau)=e^{-i(\tau-s)\mathcal{L}}u(\cdot,s)(z)
-i\displaystyle\int_{s}^{\tau} e^{-i(\tau-s_1)\mathcal{L}}G(z,s_1,u(z,s_1))ds_1.$$
Therefore we deduce from Strichartz estimates that 
\Bea \Vert u\Vert_{L^{q_1}\left((s,t),\tilde{W}_{\LL}^{1,p_1}\right)}\leq 
C\Vert u(\cdot,s)\Vert_{\tilde{W}_{\LL}^{1,2}}+C\Vert u\Vert_{L^{q_1}\left((s,t),\tilde{W}_{\LL}^{1,p_1}\right)}
\Vert u\Vert_{L^{q}\left((s,t),\tilde{W}_{\LL}^{1,p}\right)}^{\frac{2}{n-1}}.
\Eea
 where constant $C$ is independent of $s$ and $t$. 
Since $p\neq 2$, so $q<\infty$ and $u\in L^{q}\left((t_0,T^*),
\tilde{W}_{\LL}^{1,p}(\C^n)\right)$, therefore we choose $s$ sufficiently close to $T^*$ such that 
$$C\Vert u\Vert_{L^{q}\left((s,T^*),\tilde{W}_{\LL}^{1,p}(\C^n)\right)}^{\frac{2}{n-1}}\leq \frac{1}{2} .$$
Therefore we get  
$$\Vert u\Vert_{L^{q_1}\left((s,t),\tilde{W}_{\LL}^{1,p_1}(\C^n)\right)}
\leq 2C\Vert u(\cdot,s)\Vert_{\tilde{W}_{\LL}^{1,2}}.$$
Since RHS is independent of $t\in (s,T^*)$, so we have 
$u\in L^{q_1}\left((s,T^*),\tilde{W}_{\LL}^{1,p_1}(\C^n)\right)$. 
Therefore $u\in L^{q_1}\left((t_0,T^*),\tilde{W}_{\LL}^{1,p_1}\right)$ and 
$G(z,u(z,\tau))\in L^{q_1'}\left((t_0,T^*),\tilde{W}_{\LL}^{1,p_1'}\right)$ 
follows from (\ref{blowup1}).
Now from Strichartz estimates and (\ref{oprnorm1}), 
$u\in L^{\tilde{q}}((t_0,T^*),
$ $\tilde{W}_{\LL}^{1,\tilde{p}}(\C^n))\cap C([t_0,T^*],\tilde{W}_{\LL}^{1,2}(\C^n))$ 
for every admissible pair $(\tilde{q},\tilde{p})$. 
Now by considering $T^*$ as a initial time and by 
local existence argument, we get contradiction to maximality of $T^*$.

\noindent{\bf Local stability:} Let $f_k\to f$ in $\tilde{W}_{\LL}^{1,2}(\C^n)$. 
Then for each $T>0$, 
\Bea\Vert e^{-i(t-t_0)\LL}(f-f_k)\Vert_{L^{\gamma}\left(I_T,\tilde{W}_{\LL}^{1,\rho}(\C^n)\right)}
\leq C\Vert f-f_k\Vert_{\tilde{W}_{\LL}^{1,2}(\C^n)}\to 0 \mbox{~~as~~} k\to \infty
\Eea 
where $I_T=(t_0-T,t_0+T)$.
Therefore for given $\delta>0$ in (\ref{delta}), choose $T(\delta)$ 
sufficiently small such that
\bea\label{tdelta}\Vert e^{-i(t-t_0)\LL}f\Vert_{L^{\gamma}(I_T,\tilde{W}_{\LL}^{1,\rho})}\leq \frac{\delta}{2}
\eea 
and choose $k$ sufficiently large so that 
\Bea \Vert e^{-i(t-t_0)\LL}(f-f_k)\Vert_{L^{\gamma}\left(I_T,\tilde{W}_{\LL}^{1,\rho}(\C^n)\right)}
\leq C\Vert f-f_k\Vert_{\tilde{W}_{\LL}^{1,2}(\C^n)}\leq \frac{\delta}{2}.
\Eea
Therefore choose $k_0(T)$ so large such that 
\bea\label{kdelta}\Vert e^{-i(t-t_0)\LL}f_k\Vert_{L^{\gamma}\left(I_T,\tilde{W}_{\LL}^{1,\rho}(\C^n)\right)}
\leq \delta
\eea
 for $k\geq k_0(T)$. 

Let $u$ and $\tilde{u_k}$ are solutions corresponding to initial values
$f$ and $f_k$ at time $t_0$ respectively for $k\geq 1$.
In view of estimates (\ref{uclaim1}) and (\ref{uclaim2}), $u,\tilde{u_k}$ will satisfy 
following estimates
\bea\label{uclaim3} 
\Vert u\Vert_{L^{\gamma}\left(I_T,
\tilde{W}_{\LL}^{1,\rho}(\C^n)\right)}&\leq &2\delta\\
\label{uclaim4}\Vert u\Vert_{L^q\left((t_0,t_0+T),\tilde{W}_{\LL}^{1,p}(\C^n)\right)}&<&\infty\\
\label{uclaim5}\sup_{k\geq k_0(T)} \Vert \tilde{u}_k\Vert_{L^{\gamma}\left(I_T,\tilde{W}_{\LL}^{1,\rho}(\C^n)\right)}&\leq &2\delta\\
\label{uclaim6}\sup_{k\geq k_0(T)} \Vert \tilde{u}_k\Vert_{L^{q}\left(I_T,\tilde{W}_{\LL}^{1,p}(\C^n)\right)}&<&\infty
\eea 
where $(q,p)$ be any admissible pair.
Now from Strichartz estimates and Lemma \ref{estimates2},
\Bea \Vert u-\tilde{u_k}\Vert_{L^{\gamma}(I_T,L^{\rho})}&=&\Vert 
\mathcal{H}u-\mathcal{H}\tilde{u_k}\Vert_{L^{\gamma}(I_T,L^{\rho})}\\
&\leq & C\Vert f-f_k\Vert_{\tilde{W}_{\LL}^{1,2}(\C^n)}
+C\Vert G(z,u)-G(z,\tilde{u_k})\Vert_{L^{\gamma'}(I_T,L^{\rho'})}\\
&\leq & C\Vert f-f_k\Vert_{\tilde{W}_{\LL}^{1,2}(\C^n)}+
C\Vert u-\tilde{u_k}\Vert_{L^{\gamma}\left(I_T,L^{\rho}(\C^n)\right)}\times\\
&&\left(\Vert u\Vert_{L^{\gamma}\left(I_T,\tilde{W}_{\LL}^{1,\rho}(\C^n)\right)}
+\Vert \tilde{u_k}\Vert_{L^{\gamma}\left(I_T,\tilde{W}_{\LL}^{1,\rho}(\C^n)\right)} \right)^{\frac{2}{n-1}}.
\Eea
 From (\ref{delta}) and (\ref{uclaim1}), 
$$C\left(\Vert u\Vert_{L^{\gamma}\left(I_T,\tilde{W}_{\LL}^{1,\rho}(\C^n)\right)}
+\Vert \tilde{u_k}\Vert_{L^{\gamma}\left(I_T,\tilde{W}_{\LL}^{1,\rho}(\C^n)\right)} 
\right)^{\frac{2}{n-1}}\leq \frac{1}{2}.$$
Therefore $\Vert u-\tilde{u_k}\Vert_{L^{\gamma}(I_T,L^{\rho})}
\leq 2C\Vert f-f_k\Vert_{\tilde{W}_{\LL}^{1,2}(\C^n)}\to 0$ as $k\to \infty$.
Since $\{\tilde{u_k}\}$ is a bounded sequence in 
$L^{\gamma}\left(I_T,\tilde{W}_{\LL}^{1,\rho}(\C^n)\right)$, therefore 
from Lemma \ref{estimates2} with $m=0$,  
$\Vert G(z,u(z,t))
-G(z,\tilde{u_k}(z,t))\Vert_{L^{\gamma'}\left(I_T,L^{\rho'}(\C^n)\right)}\to 0$ 
as $j\to \infty$.
Since  $\mathcal{H}u=u,\mathcal{H}\tilde{u}_k=\tilde{u}_k$, therefore from Strichartz estimates, $\Vert u-\tilde{u_k}
\Vert_{L^{q}\left(I_T,L^{p}(\C^n)\right)}\to 0$ as $k\to \infty$ for every admissible pair $(q,p)$.

Note that $(\partial_{x_j}-\frac{iy_j}{2})=\frac{1}{2}(Z_j-\overline{Z}_j)$ and 
$(\partial_{y_j}+\frac{ix_j}{2})=\frac{i}{2}(Z_j+\overline{Z}_j)$.
For $S=(\partial_{x_j}-\frac{iy_j}{2}),(\partial_{y_j}+\frac{ix_j}{2})$ and using the notation 
$\psi_{(k)}= \psi\left(z,|\tilde{u}_k(z,t)|\right)$ 
(see (4.17) in \cite{pkrvks}), we have 
\begin{equation}
\begin{aligned}\label{stability}
S(G_{(k)}-G)=\psi_{(k)} S(\tilde{u}_k-u)+(\psi_{(k)}-\psi)Su+(\partial_j\psi_{(k)}) 
(\tilde{u}_k-u)\\
&\hspace{-8.6cm}+(\partial_j\psi_{(k)}-\partial_j\psi)u
+(\partial_{2n+1}\psi_{(k)}) \tilde{u}_{k}\Re(\frac{\overline{\tilde{u}_k}}{|\tilde{u}_k|}S(\tilde{u}_k-u))\\
&\hspace{-8.6cm}+(\partial_{2n+1}\psi_{(k)}) \tilde{u}_k\Re(\frac{\overline{\tilde{u}_k}}
{|\tilde{u}_k|}Su)-(\partial_{2n+1}\psi) u\Re(\frac{\overline{u}}{|u|}Su)
\end{aligned}
\end{equation}
where $\partial_j= \partial_{x_j}$ for $S=(\partial_{x_j}-\frac{iy_j}{2})$ and 
$\partial_j= \partial_{y_j}$ for $S=(\partial_{y_j}+\frac{ix_j}{2}),~1\leq j\leq n$. 

Using the assumption (\ref{nlc2}) on $\psi$,  
Lemma \ref{embedding1}, and by similar computations as used in Lemma \ref{estimates2}
and Proposition \ref{forstability1},
we have 
$$
\Vert\psi_{(k)} S(\tilde{u}_k-u)\Vert_{L^{\gamma'}\left(I_{T},L^{\rho'}\right)} \leq 
C\Vert S(\tilde{u}_k-u)\Vert_{L^{\gamma}\left(I_{T},L^{\rho}\right)} \Vert \tilde{u}_k\Vert_{L^{\gamma}\left(I_{T},\tilde{W}_{\LL}^{1,\rho}\right)}^{\frac{2}{n-1}}
$$
$$
\hspace*{-1cm} \Vert(\partial_j\psi_{(k)})(\tilde{u}_k-u)\Vert_{L^{\gamma'}\left(I_{T},L^{\rho'}\right)} 
\leq C\Vert \tilde{u}_k-u\Vert_{L^{\gamma}\left(I_{T},L^{\rho}\right)} \Vert \tilde{u}_k\Vert_{L^{\gamma}\left(I_{T},\tilde{W}_{\LL}^{1,\rho}\right)}^{\frac{2}{n-1}}
$$
\begin{equation*}
\begin{split}
&\Vert (\partial_{2n+1}\psi_{(k)}) \tilde{u}_k\Re(\frac{\overline{\tilde{u}_k}}{|\tilde{u}_k|}S(\tilde{u}_k-u)) 
\Vert_{L^{\gamma'}\left(I_{T},L^{\rho'}\right)} \\
& \hspace{5cm} \leq  C\Vert S(\tilde{u}_k-u)\Vert_{L^{\gamma}\left(I_{T},L^{\rho}\right)}
\Vert \tilde{u}_k\Vert_{L^{\gamma}\left(I_{T},\tilde{W}_{\LL}^{1,\rho}\right)}^{\frac{2}{n-1}}.
\end{split}
\end{equation*}

Since  $\Vert \tilde{u}_k-u\Vert_{L^{\gamma}(I_{T},L^{\rho}(\C^n))}\rightarrow 0$ and $\{\tilde{u}_k\}$
is a bounded sequence in $L^{\gamma}(I_{T},\tilde{W}_{\LL}^{1,\rho})$, 
therefore by second inequality in the above estimates, 
$(\partial_j\psi_{(k)})(\tilde{u}_k-u)\to 0$ as $k\to \infty$ in 
$L^{\gamma'}\left(I_{T},L^{\rho'}(\C^n)\right)$.
Since $G$ is $C^1$, so in view of the condition (\ref{nlc2}) on $\psi$ 
and Proposition \ref{forstability1}, the sequences 
$(\psi_{(k)}-\psi)Su, , (\partial_j\psi_{(k)}-\partial_j\psi)u$ and 
$(\partial_{2n+1}\psi_{(k)}) \tilde{u}_k\Re(\frac{\overline{\tilde{u}_k}}{|\tilde{u}_k|}Su)-
(\partial_{2n+1}\psi) u\Re(\frac{\overline{u}}{|u|}Su)$ converges to zero 
in $L^{\gamma'}(I_{\tau},L^{\rho'})$ as $k\rightarrow \infty$.
Using these observations in (\ref{stability}), we get
\begin{eqnarray*}
\Vert S(G_{(k)}-G)\Vert_{L^{\gamma'}(I_{T},L^{\rho'})}  \leq &\!\!\!\!\!  C\Vert \tilde{u}_k\Vert_{L^{\gamma}\left(I_{T},\tilde{W}_{\LL}^{1,\rho}\right)}^{\frac{2}{n-1}} 
\Vert S(\tilde{u}_k-u)\Vert_{{L^{\gamma}(I_{T},L^{\rho}(\C^n))}} + a_{k}
\end{eqnarray*}
where $S=(\partial_{x_j}-\frac{iy_j}{2}),(\partial_{y_j}+\frac{ix_j}{2})$ ($1\leq j\leq n$) and $a_{k}\to 0$ as $k\to \infty$. 
Since $(\partial_{x_j}-\frac{iy_j}{2})=\frac{1}{2}(Z_j-\overline{Z}_j)$ and 
$(\partial_{y_j}+\frac{ix_j}{2})=\frac{i}{2}(Z_j+\overline{Z}_j)$, therefore we have
\bea \label{mainstability}\hspace{.4cm}
\Vert G_{(k)}-G\Vert_{L^{\gamma'}(I_{T},\tilde{W}_{\LL}^{1,\rho'})}  &\leq &  C\Vert \tilde{u}_k\Vert_{L^{\gamma}\left(I_{T},\tilde{W}_{\LL}^{1,\rho}\right)}^{\frac{2}{n-1}} 
\Vert \tilde{u}_k-u\Vert_{L^{\gamma}(I_{T},\tilde{W}_{\LL}^{1,\rho})} + a_{k}.
\eea
Now from Strichartz estimates and above estimate, we have
\begin{equation}
\begin{split}\label{usefulforstability}
 \Vert \tilde{u}_k-u\Vert_{L^{\gamma}(I_{T},\tilde{W}_{\LL}^{1,\rho})} \leq C\Vert f_k-f\Vert_{\tilde{W}_{\LL}^{1,2}}
+C \Vert \tilde{u}_k\Vert_{L^{\gamma}\left(I_{T},\tilde{W}_{\LL}^{1,\rho}\right)}^{\frac{2}{n-1}} 
\Vert \tilde{u}_k-u\Vert_{L^{\gamma}(I_{T},\tilde{W}_{\LL}^{1,\rho})} + a_{k}.
\end{split}
\end{equation}
Now we choose $\delta>0$ sufficiently small  such that it satisfies condition (\ref{delta}) and 
\Bea C(2\delta)^{\frac{2}{n-1}}\leq \frac{1}{2}\Eea
where constant $C$ is appearing in the inequality (\ref{usefulforstability}). Note that $T$ depends on 
$\delta$ through (\ref{tdelta}). Therefore from estimates (\ref{uclaim5}) and (\ref{usefulforstability}), we have 
\Bea \Vert \tilde{u}_k-u\Vert_{L^{\gamma}(I_{T},\tilde{W}_{\LL}^{1,\rho})} \leq 2C\Vert f_k-f\Vert_{\tilde{W}_{\LL}^{1,2}}+2a_{k}\to 0
\Eea
as $k\to \infty$. Now from estimates (\ref{mainstability}), (\ref{uclaim3}) and (\ref{uclaim5})
\Bea\Vert G_{(k)}-G\Vert_{L^{\gamma'}(I_{T},\tilde{W}_{\LL}^{1,\rho'})} \to 0
\Eea
as $k\to \infty$.
From Strichartz estimates,  $\Vert \tilde{u}_k-u\Vert_{L^{q}(I_{T},\tilde{W}_{\LL}^{1,p})}
=\Vert \mathcal{H}\tilde{u}_k-\mathcal{H}u\Vert_{L^{q}(I_{T},\tilde{W}_{\LL}^{1,p})}$ $\to 0$ 
as $k\to \infty$
for every admissible pair $(q,p)$.

\noindent{\bf Stability:} Let $(T_{*,k},T_k^{*})$ be the maximal interval for the solutions $\tilde{u}_k$ and 
$I \subset (T_*,T^*)$  be a compact interval. 
The key idea is to extend  the local stability result proved above to the 
interval $I$ by covering it with finitely many intervals obtained by successive application of the above local stability argument. 
This is possible provided $\tilde{u}_k$ is defined on $I$, for all but finitely many $k$. 
In fact, we prove $I \subset (T_{*,k},T_{k}^*)$ for all but finitely many $k$.

Without loss of generality, we assume that $t_0\in I=[a,b]$, and give a proof by the method of contradiction. 
Suppose there exist infinitely many $T_{k_m}^* \leq b$ and let  
$c= \liminf T_{k_m}^*$. Then for $\epsilon>0, ~ [t_0, c-\epsilon] \subset [t_0,T_{k_m}^*)$ 
for all  $k_m$ sufficiently large and $\tilde{u}_{k_m}$ are defined on $ [t_0, c-\epsilon]$. 

By compactness, the local stability
result proved above can be extended to the interval $[t_0, c-\epsilon]$.

 For given $\delta>0$, 
choose $\epsilon>0$ sufficiently small such that 
\Bea \Vert e^{-i\left(t-(c-\epsilon)\right)\LL}u(\cdot,c-\epsilon)-
e^{-i\left(t-(c-\epsilon)\right)\LL}u(\cdot,c)\Vert_{L^{\gamma}\left((c-\epsilon,c+\epsilon),\tilde{W}_{\LL}^{1,\rho}\right)}\\
\leq C\Vert u(\cdot,c-\epsilon)-u(\cdot,c)\Vert_{\tilde{W}_{\LL}^{1,2}}
\leq  \frac{\delta}{6}\\
\Vert e^{-i\left(t-(c-\epsilon)\right)\LL}u(\cdot,c)-e^{-i(t-c)\LL}u(\cdot,c)\Vert_{L^{\gamma}\left((c-\epsilon,c+\epsilon),\tilde{W}_{\LL}^{1,\rho}\right)}\\
\leq C\Vert e^{-i\epsilon t\LL}u(\cdot,c)-u(\cdot,c)\Vert_{\tilde{W}_{\LL}^{1,2}}\leq  \frac{\delta}{6}\\
\Vert e^{-i(t-c)\LL}u(\cdot,c)\Vert_{L^{\gamma}\left((c-\epsilon,c+\epsilon),\tilde{W}_{\LL}^{1,\rho}\right)}\leq \frac{\delta}{6}.
\Eea
Now we choose $k_0(\epsilon)$ such that following estimate holds for all $k\geq k_0$
\Bea &\Vert e^{-i\left(t-(c-\epsilon)\right)}\tilde{u}_{k_m}(\cdot,c-\epsilon)-e^{-i\left(t-(c-\epsilon)\right)}u(\cdot,c-\epsilon)\Vert_{L^{\gamma}\left((c-\epsilon,c+\epsilon),\tilde{W}_{\LL}^{1,\rho}\right)}\\
&\leq C\Vert \tilde{u}_{k_m}(\cdot,c-\epsilon)-u(\cdot,c-\epsilon)\Vert_{\tilde{W}_{\LL}^{1,2}}\leq  \frac{\delta}{2}.
\Eea

Therefore $\Vert e^{-i\left(t-(c-\epsilon)\right)}\tilde{u}_{k_m}(\cdot,c-\epsilon)
\Vert_{L^{\gamma}\left((c-\epsilon,c+\epsilon),\tilde{W}_{\LL}^{1,\rho}\right)}\leq \delta$ 
for all $k_m\geq k_0$. Now by local existence argument (see (\ref{delta1})), $\tilde{u}_{k_m}$ is
defined on $(t_0,c+\epsilon)$  and therefore $T_{k_m}^*\geq c+\epsilon$ for 
all $k_m\geq k_0$,  hence contradicts the fact that $ \liminf T_{k_m}^*=c.$

Similarly we can show that
$[a,t_0] \subset (T_{*,k},t_0]$ for all but finitely many $k$ which completes the proof of stability.

\end{proof}

{\bf Acknowledgements:}  
Author wishes to thank the Harish-Chandra Research institute, the Dept. 
of Atomic Energy, Govt. of India, for providing excellent research facility.

\end{document}